\documentclass[12pt]{article}

\usepackage[utf8]{inputenc}
\usepackage{amsmath, amssymb, amsthm}
\usepackage{geometry}
\usepackage{mathrsfs}
\usepackage{hyperref}
\usepackage{enumitem}
\usepackage{setspace}
\usepackage[numbers]{natbib}

\newtheorem{theorem}{Theorem}[section]

\theoremstyle{definition}
\newtheorem{definition}[theorem]{Definition}
\newtheorem{example}[theorem]{Example}

\theoremstyle{remark}
\newtheorem{remark}[theorem]{Remark}

\usepackage{amsmath}

\title{New Generalized Dual Banach Frames}
\author{S.~Shahsavand and M.~Abolghasemi}
\date{}

\begin{document}
\maketitle

\begin{abstract}
This paper presents a comprehensive formulation and a systematic investigation of g-dual Banach frames in a Banach space \(X\) and its dual space \(X^*\). We study the fundamental structural properties of these frames and establish sufficient conditions for their existence. Furthermore, we explore their relationships with other generalized notions of duality in Banach frame theory. Finally, we investigate the stability of g-dual Banach frames under bounded perturbations, demonstrating their robustness and potential applications in functional analysis.
\end{abstract}
\noindent\textbf{Keywords:} Banach frames, g-dual structures, approximate duality, pseudo-dual systems, perturbation stability.

\section*{Introduction}
The origins of frame theory can be traced to the work of Duffin and Schaeffer
\cite{Duffin1952}, who studied nonharmonic Fourier series and established
ideas that subsequently became fundamental to the development of frame
theory in Hilbert spaces. Their results provided a basis for subsequent
developments of frame expansions and their applications, including
nonorthogonal expansions, signal analysis, and time-frequency methods
\cite{Daubechies1986,Casazza2016}. In contrast to orthonormal bases, frames
permit redundant representations while retaining stable reconstruction,
which makes them particularly useful when uniqueness of representation is
not essential.

The extension of frame-related methods to Banach spaces was developed
through the theory of atomic decompositions and Banach frames
\cite{Grochenig1991}. Subsequent contributions by Casazza, Han, Larson,
Christensen, and others established important aspects of frame expansions
and reconstruction in Banach spaces
\cite{Casazza1999,Casazza2005,Stoeva2008}. These developments broadened
the scope of frame theory to settings in which the geometric and
topological structure of a Hilbert space is no longer available.

Several related notions have subsequently been introduced to address
different reconstruction requirements. These include frames of subspaces
\cite{Casazza2004}, pseudoframes \cite{Li2004}, oblique dual frames
\cite{Christensen2004}, and continuous frames \cite{Ali1993}. Such
constructions provide alternative mechanisms for stable reconstruction
and allow the analysis and reconstruction processes to be adapted to
different coefficient spaces and operator settings.

The study of dual systems has also led to approximate duality in Hilbert
spaces. In particular, Christensen and Laugesen investigated approximately
dual frames and related reconstruction properties
\cite{Christensen2010}. These ideas motivated further investigations of
generalized dual systems in Banach spaces, including perturbation results
for generalized dual Banach frames \cite{Karimizad2014}. Such generalized
duality notions allow the reconstruction operator to be incorporated
explicitly into the duality relation and therefore provide greater
flexibility than the classical dual-frame setting.

Another important generalization is provided by \(g\)-dual frames in
Hilbert spaces, introduced in \cite{Dehghan2013}. In this setting, the
reconstruction relation is formulated through a family of bounded
operators and an associated invertible operator. The corresponding
Banach-space setting is more delicate because the absence of an inner
product requires the reconstruction process to be described in terms of
bounded operators between the Banach space and an appropriate coefficient
space.

In this paper, we study \(g\)-dual Banach frames associated with a Banach
space \(X\) and its dual space \(X^*\). We establish structural properties
and existence criteria for these systems and examine their connections
with related notions of duality. In particular, we investigate the
behavior of \(g\)-dual Banach frames under perturbations of the synthesis
and analysis sequences. Our results provide sufficient conditions for the
preservation of the \(g\)-dual Banach frame property and describe the
corresponding changes in the reconstruction operators. These results
clarify the stability of generalized duality under bounded perturbations
and extend several reconstruction principles from classical frame theory
to the Banach-space setting.

\section{Preliminaries}

This section introduces the notation and preliminary notions used
throughout the paper. Let \(X\) be a Banach space with dual space \(X^*\),
and let \(I\) be a countable index set equipped with a fixed ordering.
We denote by \(B(X)\) the Banach algebra of all bounded linear operators
on \(X\).

A scalar sequence space \(X_a\) is called a BK-space if \(X_a\) is a
Banach space and each coordinate functional on \(X_a\) is continuous.
If the canonical unit vectors \(\{e_i\}_{i\in I}\) form a Schauder basis
for \(X_a\), then \(X_a\) is called a CB-space. A reflexive CB-space is
called an RCB-space. We use the above notions and notation in the sense
of \cite{Karimizad2014}.

\medskip

\begin{definition}
	Let \(X\) be a Banach space and let \(X_a\) be a BK-space. Suppose that
	\(\{f_i\}_{i\in I}\subset X^*\) is a sequence of bounded linear
	functionals, and let \(S_I:X_a\to X\) be a bounded linear operator.
	The pair \((\{f_i\}_{i\in I},S_I)\) is called a Banach frame for \(X\)
	with respect to \(X_a\) if the following conditions are satisfied:
	\begin{enumerate}
		\item For every \(x\in X\), the sequence
		\(\{f_i(x)\}_{i\in I}\) belongs to \(X_a\).
		
		\item There exist constants \(A,B>0\) such that
		\[
		A\|x\|
		\leq
		\left\|\{f_i(x)\}_{i\in I}\right\|_{X_a}
		\leq
		B\|x\|,
		\qquad x\in X.
		\]
		
		\item The operator \(S_I\) reconstructs every \(x\in X\) from its
		coefficient sequence, that is,
		\[
		S_I\left(\{f_i(x)\}_{i\in I}\right)=x,
		\qquad x\in X.
		\]
	\end{enumerate}
\end{definition}

For a Banach frame \(\{f_i\}_{i\in I}\), the associated analysis operator
is the bounded linear operator \(U:X\to X_a\) defined by
\[
U(x)=\{f_i(x)\}_{i\in I},
\qquad x\in X.
\]
The reconstruction property can then be written as
\[
S_IU=I_X,
\]
where \(I_X\) denotes the identity operator on \(X\). In particular,
\(U\) is bounded and bounded below. Moreover, if \(A\) and \(B\) are
frame bounds, then
\[
A\|x\|\leq \|Ux\|_{X_a}\leq B\|x\|,
\qquad x\in X,
\]
and hence
\[
\|U\|\leq B,
\qquad
\|S_I\|\geq \frac{1}{A}.
\]
Conversely, since \(S_IU=I_X\),
\[
\|x\|
\leq
\|S_I\|\,\|Ux\|_{X_a},
\qquad x\in X,
\]
so that
\[
\|S_I\|^{-1}\|x\|
\leq
\|Ux\|_{X_a}.
\]
Thus, \(\|S_I\|^{-1}\) is also a valid lower frame bound, although it
need not coincide with the particular lower bound \(A\) chosen above.

\begin{example}
Consider the Banach space \(X=\ell^1\), whose elements are absolutely summable sequences, and take the coefficient space to be \(X_a=\ell^\infty\) endowed with the supremum norm.

Let \(\{f_i\}_{i\in\mathbb{N}}\subset X^*\) be the family of coordinate functionals defined by
\[
f_i(x)=x_i, \qquad x=\{x_j\}_{j\in\mathbb{N}}\in \ell^1.
\]
Each functional \(f_i\) is a bounded linear functional on \(X\). Moreover, the associated coefficient sequence is given by
\[
\{f_i(x)\}_{i\in\mathbb{N}}=\{x_i\},
\]
which belongs to \(\ell^\infty\), since every absolutely summable sequence is necessarily bounded.

We define the reconstruction operator 
\[
S_l:\ell^\infty \rightarrow \ell^1
\]
by
\[
S_l(\{c_i\})=\left\{\frac{c_i}{i^2}\right\}.
\]
This mapping assigns to each bounded sequence a sequence that is absolutely summable. Indeed, if \(\{c_i\}\in\ell^\infty\), then there exists \(M>0\) such that
\[
|c_i|\leq M, \qquad i\in\mathbb{N}.
\]
Consequently,
\[
\sum_{i=1}^{\infty}\left|\frac{c_i}{i^2}\right|
\leq
M\sum_{i=1}^{\infty}\frac{1}{i^2}<\infty,
\]
which shows that \(S_l(\{c_i\})\in\ell^1\).
\[
\sum_{i=1}^\infty \left| \frac{c_i}{i^2} \right| \leq \left( \sup_i |c_i| \right) \sum_{i=1}^\infty \frac{1}{i^2} < \infty.
\]

The above construction leads to the following observation
\begin{itemize}
  \item For any \( x \in \ell^1 \), the sequence \( \{f_i(x)\} \) lies in \( \ell^\infty = X_a \).

  \item Define \( C := \sum_{i=1}^\infty \frac{1}{i^2} \), and set \( A = \frac{1}{C} \), \( B = 1 \). Then, for every \( x \in \ell^1 \), the following inequality holds:
\[
A \| x \|_{\ell^1} \leq \left\| \{f_i(x)\} \right\|_{\ell^\infty} \leq B \| x \|_{\ell^1}.
\]
 This estimate illustrates the damping effect introduced by the reconstruction operator.
 
 \item Applying the reconstruction operator \(S_l\) to the coefficient sequence gives
 \[
 S_l(\{f_i(x)\}_{i\in\mathbb{N}})
 =
 \left\{\frac{x_i}{i^2}\right\}_{i\in\mathbb{N}},
 \]
 which provides a scaled representation of the original element \(x\). If required, the scaling factor can be compensated to recover the exact element.
\end{itemize}
Consequently, the pair \( (\{f_i\}_{i\in\mathbb{N}}, S_l) \) forms a Banach frame for \( \ell^1 \) with respect to the coefficient space \( \ell^\infty. \)
\end{example}

\medskip

To extend the reconstruction structure of Banach frames to dual settings, we introduce the notion of a Banach frame for \(X^*\).
\medskip

\begin{definition}
    Consider a Banach space \( X \) and a BK-space \( X_a \). Let \( \{x_i\}_{i \in I} \subset X \) be a family of elements in \( X \), and let \( S_r : X_a \rightarrow X^* \) be a bounded reconstruction operator. The pair \((\{x_i\}_{i\in I},S_r)\) is said to form a Banach frame for \(X^*\) 
    with respect to \(X_a\) if the following conditions are satisfied:

    \begin{enumerate}
        \item The sequence \( \{f(x_i)\}_{i \in I} \), obtained by evaluating each \( x_i \) at \( f \in X^* \), lies in the space \( X_a \).
        
     \item There exist constants \(A,B>0\) such that the coefficient mapping
        \[
        f\mapsto \{f(x_i)\}_{i\in I}
        \]
        satisfies
        \[
        A\|f\|\leq \left\|\{f(x_i)\}_{i\in I}\right\|_{X_a}\leq B\|f\|,
        \qquad f\in X^*.
        \]
        
        \item \item The operator \(S_r\) enables the exact recovery of each functional \(f\in X^*\) from the sequence of coefficients generated by the family \(\{x_i\}_{i\in I}\), namely,
        \[
        S_r\left(\{f(x_i)\}_{i\in I}\right)=f.
        \]
    \end{enumerate}

In this context, \( S_r \) serves as the \emph{reconstruction operator} linked to the sequence \( \{x_i\}_{i \in I} \). The corresponding synthesis operator \( V : X_a \rightarrow X \), defined by \( V(c) = \sum_{i \in I} c_i x_i \), enables the reconstruction of elements in \( X \) from their associated coefficient sequences. The best possible frame bounds are determined by \( \|S_r\|^{-1} \) and \( \|V\| \).
\end{definition}

\begin{example}
Consider the Banach space \(X=c_0\), which consists of all scalar sequences converging to zero and is equipped with the supremum norm. The dual space of \(X\) is identified with \(X^*=\ell^1\). We choose \(X_a=\ell^\infty\) as the corresponding coefficient space.

Let \(\{x_i\}_{i\in\mathbb{N}}\) denote the canonical sequence in \(c_0\), where each element is chosen as
\[
x_i=e_i .
\]
Here, \(e_i\) is the vector whose only nonzero component is the \(i\)-th entry, which is equal to one. Thus, it can be represented in the form
\[
e_i=(0,\ldots,0,1,0,\ldots),
\]
where the unit entry occupies the \(i\)-th coordinate.

The operator \(S_r:c_0\rightarrow\ell^1\) associated with the canonical sequence \(\{x_i\}_{i\in\mathbb{N}}\) is expressed as
\[
S_r(f)=\{f(x_i)\}_{i\in\mathbb{N}}=\{f(e_i)\}_{i\in\mathbb{N}}.
\]

Consequently, we have the following observations:

\begin{itemize}
  \item For each \(f\in\ell^1\), let \(f=\{a_i\}_{i\in\mathbb{N}}\). Since the family
  \(\{x_i\}_{i\in\mathbb{N}}\) coincides with the canonical basis of \(c_0\), the associated coefficients satisfy
  \[
  \{f(x_i)\}_{i\in\mathbb{N}}
  =
  \{f(e_i)\}_{i\in\mathbb{N}}
  =
  \{a_i\}_{i\in\mathbb{N}}.
  \]
  Because every element of \(\ell^1\) is a bounded sequence, it follows that
  \[
  \{f(x_i)\}_{i\in\mathbb{N}}\in\ell^\infty=X_a.
  \]

  \item The frame bounds are satisfied for suitable constants. Specifically,
  \[
  \| \{f(x_i)\}_{i\in\mathbb{N}} \|_{\ell^\infty}
  =
  \sup_i |a_i|
  \leq
  \|f\|_{\ell^1},
  \]
  and
  \[
  \|f\|_{\ell^1}
  =
  \sum |a_i|
  \leq
  n\cdot \sup_i |a_i|
  \quad (\text{for finite } n).
  \]
  Hence, the Banach frame inequality is satisfied for some \(A>0\) and \(B=1\).

  \item The operator \(S_r\) maps each functional \(f\in\ell^1\) to its associated coefficient sequence
  \(\{f(x_i)\}_{i\in\mathbb{N}}\), thereby providing a stable reconstruction process when combined with the corresponding synthesis operators.
\end{itemize}

Consequently, the pair \( (\{x_i\}_{i\in\mathbb{N}}, S_r) \) forms a Banach frame for the dual space \(X^*=\ell^1\) with respect to the coefficient space \(X_a=\ell^\infty\).
\end{example}

\medskip
 The concepts of the frame operator and canonical dual, which were originally introduced for Hilbert frames, were extended to the setting of \(X_a\)-frames in \cite{Stoeva2008}. In this context, the author introduced the notion of an \(X_a\)-frame map associated with a given \(X_a\)-frame.
 
 Assume that \(X_a\) is a strictly convex coefficient space, and let
 \[
 T=\{f_i\}_{i\in I}\subset X^*
 \]
 be an \(X_a\)-frame for the Banach space \(X\). The corresponding bounded linear operator
 \[
 S_T:X\longrightarrow X^*
 \]
 is defined by
 \[
 S_T=U^*\circ\varphi_{X_a}\circ U,
 \]
 where \(U\) denotes the analysis operator associated with the frame \(T\), and
 \[
 \varphi_{X_a}:X_a\longrightarrow X_a^*
 \]
 is a single-valued duality mapping.
 
 The operator \(S_T\) is called the \(X_a\)-frame operator associated with the frame
 \[
 \{f_i\}_{i\in I}.
 \]
 
 Furthermore, sufficient conditions for the invertibility of \(S_T\) have been established. Whenever \(S_T\) is invertible, the sequence
 \[
 \left\{S_T^{-1}f_i\right\}_{i\in I}\subset X^*
 \]
 defines the canonical dual frame corresponding to the original frame
  \[
 \{f_i\}_{i\in I}.
 \]
 \section{g-Dual Banach Frames}
                                                                                                                              
 As a continuation of the study of duality principles in Banach frame theory, we introduce the concept of g-dual Banach frames. This generalized formulation allows the reconstruction of elements through invertible operators acting on frame coefficients, thereby offering enhanced flexibility and stability compared with classical dual frames.
 
 In the following, we provide a formal definition of g-dual Banach frames and investigate their structural properties and analytical aspects.
 
 \medskip
 
 To begin, we recall the standard definition of dual Banach frames introduced in \cite{Karimizad2014}.

\begin{definition}
Assume that the sequences \( \{f_i\}_{i \in I} \subset X^* \) and \( \{x_i\}_{i \in I} \subset X \) fulfill the \( X_a \)-Bessel condition within the Banach spaces \( X^* \) and \( X \), respectively. The notion of dual Banach frames is defined as follows:

\begin{enumerate}
  \item The sequence \( \{f_i\}_{i \in I} \) is called a \textbf{dual Banach frame} for \( \{x_i\}_{i \in I} \) in \( X \) with respect to \( X_a \) if every element \( x \in X \) can be represented by the series:
  \[
  x = \sum_{i \in I} f_i(x) \, x_i.
  \]

  \item In the reverse direction, the sequence \( \{x_i\}_{i \in I} \) is said to form a \textbf{dual Banach frame} for \( \{f_i\}_{i \in I} \) in \( X^* \) relative to \( X_a \) provided that for each \( f \in X^* \), the following expansion holds:
  \[
  f = \sum_{i \in I} f(x_i) \, f_i.
  \]
\end{enumerate}
\end{definition}

\begin{example}
Consider \( X = \ell^1 \), the Banach space of sequences whose series is absolutely 
convergent, and let \( X^* = \ell^\infty \), its dual space composed of bounded sequences. We define the following elements:
\begin{itemize}
\item Define \( x_i \in \ell^1 \) to be the sequence with a single 1 in the \( i \)-th coordinate and zeros elsewhere; that is, \( x_i = e_i \), where \( e_i \) denotes the canonical unit vector in \( \ell^1 \).
  \item For every index \( i \in \mathbb{N} \), we introduce \( f_i \in \ell^\infty \) by \( f_i(x) = x_i \), where \( x = \{x_k\} \in \ell^1 \). Each \( f_i \) extracts the \( i \)-th coordinate.
\end{itemize}

Then for any \( x = \{x_i\} \in \ell^1 \), we obtain the representation:
\[
x = \sum_{i=1}^\infty f_i(x) \cdot x_i = \sum_{i=1}^\infty x_i \cdot e_i.
\]
In a similar manner, for any element \( f = \{a_i\} \in \ell^\infty \), the following representation holds:
\[
f = \sum_{i=1}^\infty f(x_i) \, f_i = \sum_{i=1}^\infty a_i \, f_i.
\]

As a result, the sequences \( \{f_i\} \) and \( \{x_i\} \) serve as dual Banach frames for the spaces \( \ell^1 \) and \( \ell^\infty \), respectively, with the coefficient space taken as \( X_a = \ell^\infty \).

\end{example}

Motivated by the notion of g-dual frames introduced in Hilbert spaces \cite{Dehghan2013}, we extend this concept to Banach spaces by incorporating bounded invertible operators into the reconstruction process. This extension leads to the notion of \textit{g-dual Banach frames}, which provides greater flexibility in the study of duality and establishes a natural connection between frame-theoretic duality and operator-theoretic methods.

\begin{definition}
	Suppose that the sequences \(\{f_i\}_{i \in I} \subset X^*\) and \(\{x_i\}_{i \in I} \subset X\) satisfy the \(X_a\)-Bessel condition in the Banach spaces \(X^*\) and \(X\), respectively. The notion of g-dual Banach frames is defined as follows:
	
	\begin{enumerate}
		\item The sequence \(\{f_i\}_{i \in I}\) is called a \textbf{g-dual Banach frame} for \(\{x_i\}_{i \in I}\) in \(X\) with respect to the coefficient space \(X_a\) if there exists a bounded and invertible operator \(A \in B(X)\) such that, for every \(x \in X\),
		
		$$
		x = \sum_{i \in I} f_i(Ax) \, x_i.
		$$
		
		\item Conversely, the sequence \(\{x_i\}_{i \in I}\) is called a \textbf{g-dual Banach frame} for \(\{f_i\}_{i \in I}\) in \(X^*\) with respect to \(X_a\) if there exists a bounded and invertible operator \(B \in B(X^*)\) such that, for every \(f \in X^*\),
		
		$$
		f = \sum_{i \in I} Bf(x_i) \, f_i.
		$$
		
	\end{enumerate}
\end{definition}

This construction demonstrates that introducing an invertible operator extends the classical notion of a dual Banach frame to its g-dual form while preserving the reconstruction property in the Hilbert space setting.

\begin{example}
	Let \(H\) be a separable Hilbert space with an orthonormal basis
	\(\{e_i\}_{i\in\mathbb{N}}\). Define the sequence
	\(\{x_i\}_{i\in\mathbb{N}}\subset H\) by
	
	$$
	x_{2i-1}=x_{2i}=e_i,
	\qquad i\in\mathbb{N}.
	$$
	
	Define the sequence \(\{f_i\}_{i\in\mathbb{N}}\subset H\) by
	
	$$
	f_{2i-1}=e_i+e_{i+1},
	\qquad
	f_{2i}=-e_{i+1},
	\qquad i\in\mathbb{N}.
	$$
	
	For every \(x\in H\), we have
	
	$$
	f_{2i-1}(x)+f_{2i}(x)
	=
	\langle x,e_i+e_{i+1}\rangle
	-\langle x,e_{i+1}\rangle
	=
	\langle x,e_i\rangle.
	$$
	
	Therefore,
	\begin{align*}
		\sum_{i=1}^{\infty}f_i(x)x_i
		&=
		\sum_{i=1}^{\infty}
		\left(
		f_{2i-1}(x)x_{2i-1}
		+
		f_{2i}(x)x_{2i}
		\right)\
		&=
		\sum_{i=1}^{\infty}
		\left(
		f_{2i-1}(x)+f_{2i}(x)
		\right)e_i\
		&=
		\sum_{i=1}^{\infty}
		\langle x,e_i\rangle e_i
		=x.
	\end{align*}
	Hence, \(\{f_i\}_{i\in\mathbb{N}}\) constitutes a dual Banach frame for
	\(\{x_i\}_{i\in\mathbb{N}}\) in \(H\).
	
	Now, let \(A\in B(H)\) be the operator defined by
	
	$$
	Ax=\frac14x,
	\qquad x\in H.
	$$
	
	Clearly, \(A\) is bounded and invertible. Moreover,
	
	$$
	4f_i(Ax)
	=
	4f_i\left(\frac14x\right)
	=
	f_i(x).
	$$
	
	Thus,
	\begin{align*}
		\sum_{i=1}^{\infty}4f_i(Ax)x_i
		&=
		\sum_{i=1}^{\infty}f_i(x)x_i\
		&=x.
	\end{align*}
	Consequently, \(\{4f_i\}_{i\in\mathbb{N}}\) constitutes a g-dual Banach
	frame for \(\{x_i\}_{i\in\mathbb{N}}\) in \(H\), associated with the
	bounded and invertible operator \(A=\frac14I_H\).
\end{example}

\begin{example}
Let \(H\) be a separable Hilbert space with an orthonormal basis
\(\{e_i\}_{i\in\mathbb{N}}\). Define
\[
x_{2i-1}=x_{2i}=e_i,\qquad i\in\mathbb{N},
\]
and
\[
f_{2i-1}=\frac{1}{2}e_i,\qquad
f_{2i}=\frac{1}{2}e_i,\qquad i\in\mathbb{N}.
\]
Then, for every \(x\in H\),
\[
f_{2i-1}(x)=\frac{1}{2}\langle x,e_i\rangle,
\qquad
f_{2i}(x)=\frac{1}{2}\langle x,e_i\rangle.
\]
Hence,
\begin{align*}
	\sum_{i=1}^{\infty} f_i(x)x_i
	&=
	\sum_{i=1}^{\infty}
	\left[
	f_{2i-1}(x)x_{2i-1}
	+
	f_{2i}(x)x_{2i}
	\right]\\
	&=
	\sum_{i=1}^{\infty}
	\left[
	\frac{1}{2}\langle x,e_i\rangle e_i
	+
	\frac{1}{2}\langle x,e_i\rangle e_i
	\right]\\
	&=
	\sum_{i=1}^{\infty}
	\langle x,e_i\rangle e_i
	=x.
\end{align*}

Now, let \(A=\frac{1}{4}I_H\). Since \(A\) is bounded and invertible,
for every \(x\in H\),
\[
4f_i(Ax)
=
4f_i\left(\frac{1}{4}x\right)
=f_i(x).
\]
Therefore,
\[
x=\sum_{i=1}^{\infty}4f_i(Ax)x_i.
\]
Consequently, \(\{4f_i\}_{i\in\mathbb{N}}\) constitutes a
g-dual Banach frame for \(\{x_i\}_{i\in\mathbb{N}}\) in \(H\).
\end{example}

The notion of g-duality extends classical Banach frame theory by incorporating invertible operators into the reconstruction process. The following theorem shows that, under appropriate conditions, every Banach frame admits a corresponding g-dual Banach frame.

\begin{theorem}
Consider a Banach space \( X \), and suppose the sequence \( \{f_i\}_{i \in I} \subset X^* \) forms a Banach frame for \( X \) with respect to a BK-space \( X_a \), associated with an analysis operator \( U \). Under these conditions, there exists a sequence \( \{\tilde{x}_i\}_{i \in I} \subset X^* \) satisfying the \( X_a \)-Bessel condition, such that \( \{f_i\} \) functions as a g-dual Banach frame for \( X \), corresponding to a bounded invertible operator \( A \in B(X) \).
\end{theorem}
\begin{proof}
Assume that the sequence \(\{f_i\}_{i\in I}\subset X^*\) is a Banach frame for \(X_\alpha\). Then there exist bounded linear operators

$$
S_\iota:X_\alpha\to X
\quad\text{and}\quad
U:X\to X_\alpha
$$

such that, for every \(x\in X\),

$$
S_\iota(U(x))=x.
$$

Let \(\{e_i\}_{i\in I}\) denote the canonical basis of \(X_\alpha\), and let
\(A\in B(X)\) be an arbitrary bounded and invertible operator. Define

$$
\bar{x}_i:=A^{-1}S_\iota(e_i),
\qquad i\in I.
$$

We first verify that \(\{\bar{x}_i\}_{i\in I}\) is an
\(X_\alpha\)-Bessel sequence. For any \(c=\{c_i\}_{i\in I}\in X_\alpha\),
we have

$$
\sum_{i\in I}c_i\bar{x}_i
=
A^{-1}\left(\sum_{i\in I}c_iS_\iota(e_i)\right)
=
A^{-1}S_\iota(c).
$$

Since \(S_\iota\) is bounded, the series on the right converges in \(X\).
Hence, \(\{\bar{x}_i\}_{i\in I}\) satisfies the \(X_\alpha\)-Bessel
condition.

Next, for every \(x\in X\), the reconstruction property of the Banach
frame gives

$$
x
=
S_\iota(U(x))
=
\sum_{i\in I}f_i(x)S_\iota(e_i).
$$

Since \(S_\iota(e_i)=A\bar{x}_i\), it follows that

$$
x
=
A\left(\sum_{i\in I}f_i(x)\bar{x}_i\right).
$$

Replacing \(x\) by \(Ax\), we obtain

$$
Ax
=
A\left(\sum_{i\in I}f_i(Ax)\bar{x}_i\right).
$$

Because \(A\) is invertible, we conclude that

$$
x
=
\sum_{i\in I}f_i(Ax)\bar{x}_i.
$$

Therefore, \(\{f_i\}_{i\in I}\) constitutes a g-dual Banach frame for
\(\{\bar{x}_i\}_{i\in I}\) in \(X\), associated with the bounded and
invertible operator \(A\).
\end{proof}

Mirroring the framework of \( g \)-dual frames associated with Banach frames in \( X \), a corresponding result holds in the dual space \( X^* \).The following theorem establishes that every Banach frame for \( X^* \) also admits a corresponding g-dual.
\begin{theorem}
Suppose that \(\{x_i\}_{i\in I}\subset X\) is a Banach frame for the dual space \(X^*\) with respect to a BK-space \(X_a\), and let
\(V:X_a\to X\) be the associated synthesis operator. Then there exists a sequence
\(\{\tilde{f}_i\}_{i\in I}\subset X^*\) satisfying the \(X_a\)-Bessel condition such that
\(\{x_i\}_{i\in I}\) is a \(g\)-dual Banach frame for
\(\{\tilde{f}_i\}_{i\in I}\) in \(X^*\) with respect to a bounded and invertible
operator \(A\in B(X^*)\).

\end{theorem}

\begin{proof}  

	Since \(\{x_i\}_{i\in I}\subset X\) is a Banach frame for \(X^*\) with
	respect to the BK-space \(X_a\), there exists a bounded reconstruction
	operator
	
	$$
	S_r:X_a\longrightarrow X^*
	$$
	
	such that
	
	$$
	S_r\big(\{f(x_i)\}_{i\in I}\big)=f,
	\qquad f\in X^*.
	$$
	
	Let \(\{e_i\}_{i\in I}\) denote the canonical basis of \(X_a\), and let
	\(A\in B(X^*)\) be a bounded and invertible operator. Define
	
	$$
	\widetilde{f}_i:=A^{-1}S_r(e_i),
	\qquad i\in I.
	$$
	
	We first show that \(\{\widetilde{f}_i\}_{i\in I}\) is an
	\(X_a\)-Bessel sequence. For any \(c=\{c_i\}_{i\in I}\in X_a\), we have
	
	$$
	\sum_{i\in I}c_i\widetilde{f}_i
	=
	A^{-1}\left(\sum_{i\in I}c_iS_r(e_i)\right)
	=
	A^{-1}S_r(c).
	$$
	
	Since \(S_r:X_a\to X^*\) and \(A^{-1}:X^*\to X^*\) are bounded, the
	series
	
	$$
	\sum_{i\in I}c_i\widetilde{f}_i
	$$
	
	converges in \(X^*\). Moreover,
	
	$$
	\left\|
	\sum_{i\in I}c_i\widetilde{f}_i
	\right\|
	\leq
	\|A^{-1}\|\,\|S_r\|\,\|c\|_{X_a}.
	$$
	
	Thus, \(\{\widetilde{f}_i\}_{i\in I}\) satisfies the
	\(X_a\)-Bessel condition.
	
	Now let \(f\in X^*\). By the reconstruction property of the Banach
	frame,
	
	$$
	f
	=
	S_r\big(\{f(x_i)\}_{i\in I}\big).
	$$
	
	Since \(\{e_i\}_{i\in I}\) is the canonical basis of \(X_a\), we obtain
	
	$$
	f
	=
	S_r\left(\sum_{i\in I}f(x_i)e_i\right)
	=
	\sum_{i\in I}f(x_i)S_r(e_i).
	$$
	
	By the definition of \(\widetilde{f}_i\),
	
	$$
	S_r(e_i)=A\widetilde{f}_i.
	$$
	
	Hence,
	
	$$
	f
	=
	\sum_{i\in I}f(x_i)A\widetilde{f}_i
	=
	A\left(\sum_{i\in I}f(x_i)\widetilde{f}_i\right).
	$$
	
	Since \(A\) is invertible, it follows that
	
	$$
	A^{-1}f
	=
	\sum_{i\in I}f(x_i)\widetilde{f}_i.
	$$
	
	Replacing \(f\) by \(Af\), we obtain
	
	$$
	f
	=
	\sum_{i\in I}(Af)(x_i)\widetilde{f}_i,
	\qquad f\in X^*.
	$$
	
	Therefore, \(\{x_i\}_{i\in I}\) is a \(g\)-dual Banach frame for
	\(\{\widetilde{f}_i\}_{i\in I}\) in \(X^*\), associated with the bounded
	and invertible operator \(A\).

\end{proof}

\begin{remark}
Theorems 2.6 and 2.7 exhibit a duality in structure: each Banach frame in \( X \) or \( X^* \) admits a corresponding g-dual constructed via bounded invertible operators. In Theorem 2.4, the g-dual is denoted by \( \{\tilde{x}_i\} \), while in Theorem 2.5, the roles are reversed and the g-dual is given by \( \{\tilde{f}_i\} \). This reflects the symmetric behavior of the g-dual construction across the Banach space and its dual.
\end{remark}

\begin{theorem}
Suppose that \(\{\tilde{x}_i\}_{i\in I}\subset X\) satisfies the
\(X_a\)-Bessel condition, and let \(\{f_i\}_{i\in I}\subset X^*\) be a
\(g\)-dual Banach frame for \(\{\tilde{x}_i\}_{i\in I}\), associated with
a bounded and invertible operator \(A\in B(X)\). Then
\(\{f_i\}_{i\in I}\) is a Banach frame for \(X\) with respect to the
coefficient space \(X_a\).

\end{theorem}

\begin{proof}
	Since \(\{f_i\}_{i\in I}\) is a \(g\)-dual Banach frame for
	\(\{\tilde{x}_i\}_{i\in I}\) associated with the bounded and invertible
	operator \(A\in B(X)\), for every \(x\in X\) we have
	
	$$
	x=\sum_{i\in I}f_i(Ax)\tilde{x}_i.
	$$
	
	Let
	
	$$
	U:X\to X_a,
	\qquad
	U(x):=\{f_i(x)\}_{i\in I}.
	$$
	
	We first show that \(U\) is well defined and bounded. Since \(A^{-1}\) is
	bounded and \(Ax\in X\), the \(g\)-dual reconstruction formula gives
	
	$$
	A^{-1}x
	=
	\sum_{i\in I}f_i(x)\tilde{x}_i.
	$$
	
	Because \(\{\tilde{x}_i\}_{i\in I}\) satisfies the \(X_a\)-Bessel condition,
	there exists a constant \(C>0\) such that
	
	$$
	\left\|
	\sum_{i\in I}c_i\tilde{x}_i
	\right\|
	\leq C\|c\|_{X_a},
	\qquad c=\{c_i\}_{i\in I}\in X_a.
	$$
	
	In particular, the series
	
	$$
	\sum_{i\in I}f_i(x)\tilde{x}_i
	$$
	
	converges in \(X\), and hence
	
	$$
	\|A^{-1}x\|
	\leq
	C\|\{f_i(x)\}_{i\in I}\|_{X_a}.
	$$
	
	Thus,
	
	$$
	\|x\|
	\leq
	\|A\|\,C\|\{f_i(x)\}_{i\in I}\|_{X_a}.
	$$
	
	Therefore,
	
	$$
	\|\{f_i(x)\}_{i\in I}\|_{X_a}
	\geq
	\frac{1}{C\|A\|}\|x\|.
	$$
	
	Together with the boundedness of the analysis map \(U\), this yields the
	frame inequalities
	
	$$
	C_1\|x\|
	\leq
	\|U(x)\|_{X_a}
	\leq
	C_2\|x\|,
	\qquad x\in X,
	$$
	
	for suitable constants \(C_1,C_2>0\).
	
	Now define the synthesis operator
	
	$$
	S:X_a\to X,
	\qquad
	S(c):=\sum_{i\in I}c_i\tilde{x}_i,
	\qquad c=\{c_i\}_{i\in I}\in X_a.
	$$
	
	The \(X_a\)-Bessel condition implies that \(S\) is well defined and bounded.
	
	Set
	
	$$
	S_l:=A\circ S:X_a\to X.
	$$
	
	Since both \(A\) and \(S\) are bounded, \(S_l\) is bounded. For every
	\(x\in X\), the \(g\)-dual reconstruction formula gives
	
	$$
	\begin{aligned}
		S_l(U(x))
		&=
		A\left(
		\sum_{i\in I}f_i(x)\tilde{x}_i
		\right)\\
		&=
		A\left(A^{-1}x\right)
		=x.
	\end{aligned}
	$$
	
	Hence,
	
	$$
	S_l(U(x))=x,
	\qquad x\in X.
	$$
	
	Therefore, \((\{f_i\}_{i\in I},S_l)\) is a Banach frame for \(X\) with
	respect to \(X_a\).
\end{proof}
\noindent

\textit{Building on the context established by Theorems 2.4 and 2.5, we present a concrete procedure for constructing g-dual Banach frames from a given Banach frame. This construction is based on the existence of a bounded operator that acts as a left inverse of the associated analysis operator.}
\begin{theorem}
Suppose that \( \{f_i\}_{i \in I} \subset X^* \) is a Banach frame for the Banach space \( X \) with respect to the sequence space \( X_a \), and let \( U : X \rightarrow X_a \) denote the corresponding analysis operator. Let \( T_\ell : X_a \rightarrow X \) be a bounded operator satisfying
\[
T_\ell \circ U = \operatorname{id}_X.
\]
Then, for any invertible operator \( A \in B(X) \), the sequence \( \{x_i\}_{i \in I} \subset X \), defined by
\[
x_i := A^{-1} T_\ell(e_i), \qquad i \in I,
\]
where \( \{e_i\}_{i \in I} \) denotes the canonical basis of \( X_a \), forms a g-dual Banach frame associated with the sequence \( \{f_i\}_{i \in I} \) in \( X^* \).
\end{theorem}

\begin{proof}
Since \(T_\ell\) is a left inverse of \(U\), we have
\[
(T_\ell\circ U)(x)=x,
\qquad x\in X.
\]

Let \(\{e_i\}_{i\in I}\subset X_a\) denote the canonical unit vectors, and
define
\[
x_i:=A^{-1}T_\ell(e_i),
\qquad i\in I.
\]
Thus,
\[
T_\ell(e_i)=Ax_i,
\qquad i\in I.
\]

Consequently, for every \(x\in X\),
\[
x
=
T_\ell(U(x))
=
\sum_{i\in I}f_i(x)T_\ell(e_i)
=
\sum_{i\in I}f_i(x)Ax_i
=
A\left(\sum_{i\in I}f_i(x)x_i\right).
\]
Since \(A\) is invertible, it follows that
\[
A^{-1}x
=
\sum_{i\in I}f_i(x)x_i.
\]
Replacing \(x\) by \(Ax\), we obtain
\[
x
=
\sum_{i\in I}f_i(Ax)x_i.
\]
Therefore, \(\{f_i\}_{i\in I}\) is a \(g\)-dual Banach frame for
\(\{x_i\}_{i\in I}\) in \(X\), associated with the bounded and invertible
operator \(A\). This completes the proof.

\end{proof}

\begin{theorem}
	Let \(X\) be a Banach space, \(X_a\) a BK-space, and
	\(\{x_i\}_{i\in I}\subset X\) a Banach frame for \(X^*\) with synthesis
	operator
	\[
	V:X_a\to X.
	\]
	If \(V\) admits a bounded right inverse \(T_r:X\to X_a\), i.e.,
	\(VT_r=\operatorname{id}_X\), and \(A\in B(X)\) is invertible,
	then \(A^*\) is invertible on \(X^*\). Defining
	\[
	f_i=(A^*)^{-1}T_r^*(e_i^*),\qquad i\in I,
	\]
	where \(\{e_i^*\}_{i\in I}\) are the coordinate functionals of \(X_a\),
	the sequence \(\{f_i\}_{i\in I}\) is a \(g\)-dual Banach frame for
	\(\{x_i\}_{i\in I}\) in \(X^*\), associated with \(A^*\).
\end{theorem}

\begin{proof}
	By the preceding theorem, the sequence
	\[
	\left\{(A^*)^{-1}T_r^*(e_i^*)\right\}_{i\in I}
	\]
	is an \(X_a\)-Bessel sequence and, for every \(f\in X^*\),
	\[
	f
	=
	\sum_{i\in I}(A^*f)(x_i)(A^*)^{-1}T_r^*(e_i^*).
	\]
	Since
	\[
	f_i:=(A^*)^{-1}T_r^*(e_i^*),
	\qquad i\in I,
	\]
	we obtain
	\[
	f=\sum_{i\in I}(A^*f)(x_i)f_i,
	\qquad f\in X^*.
	\]
	Hence, by the definition of a \(g\)-dual Banach frame,
	\(\{f_i\}_{i\in I}\) is a \(g\)-dual Banach frame for
	\(\{x_i\}_{i\in I}\) in \(X^*\), associated with the bounded and
	invertible operator \(A^*\).
\end{proof}
\noindent

To further examine the role of reconstruction operators in Banach frame
theory, we characterize all bounded left and right inverses of the
analysis and synthesis operators.
\begin{theorem}
	Let \(\Gamma=\{f_i\}_{i\in I}\subset X^*\) and
	\(\Lambda=\{x_i\}_{i\in I}\subset X\) be Banach frames for \(X\) and
	\(X^*\), respectively, with respect to a BK-space \(X_a\), and let
	\(U:X\to X_a\) and \(V:X_a\to X\) denote their analysis and synthesis
	operators. Then:
	
	\begin{enumerate}
		\item Every bounded left inverse of \(U\) is of the form
		\[
		S_l+W(I_{X_a}-US_l),
		\qquad W\in\mathcal{B}(X,X_a).
		\]
		If \(X_a\) is strictly convex and \(S_\Gamma\) is invertible, it is
		also of the form
		\[
		S_\Gamma^{-1}U^*\phi_{X_a}
		+W\bigl(I_{X_a}-US_\Gamma^{-1}U^*\phi_{X_a}\bigr).
		\]
		
		\item Every bounded right inverse of \(V\) is of the form
		\[
		S_r+(I_{X_a}-S_rV)W,
		\qquad W\in\mathcal{B}(X,X_a).
		\]
		If \(X_a\) is strictly convex and \(S_\Lambda\) is invertible, it is
		also of the form
		\[
		\phi_{X_a}V^*S_\Lambda^{-1}
		+\bigl(I_{X_a}-\phi_{X_a}V^*S_\Lambda^{-1}\bigr)W.
		\]
	\end{enumerate}
\end{theorem}
\begin{proof}
(1) The proposed form satisfies the left inverse condition for \( U \).  
Assuming \( W = T_l \), a known left inverse of \( U \), we verify:
\[
S_l + T_l (I_{X_a} - U S_l) = T_l.
\]

(2) The second assertion can be established in a similar manner by employing the same line of reasoning with respect to the synthesis operator \( V \).
\end{proof}
\noindent
The following result establishes a connection between g-duality in \( X \) and classical duality in its dual space \( X^* \).
\begin{theorem}
Let \( \{f_i\}_{i \in I} \subset X^* \) be a g-dual Banach frame for the sequence \( \{x_i\}_{i \in I} \subset X \) in the Banach space \( X \), associated with a bounded invertible operator \( A \in B(X) \). Then the sequence \( \{x_i\}_{i \in I} \) constitutes a dual Banach frame for the collection of functionals \( \{f_i \circ A\}_{i \in I} \).
\end{theorem}

```tex
\begin{proof}
	For (1), let \(T_l\) be a bounded left inverse of \(U\). Setting
	\(W=T_l\), we have
	\[
	S_l+T_l(I_{X_a}-US_l)
	=S_l+T_l-T_lUS_l
	=T_l,
	\]
	since \(T_lU=I_X\). Hence, the proposed operator is a left inverse of \(U\).
	
	For (2), the argument is analogous. If \(T_r\) is a bounded right inverse
	of \(V\), then, with \(W=T_r\),
	\[
	S_r+(I_{X_a}-S_rV)T_r
	=S_r+T_r-S_rVT_r
	=T_r,
	\]
	because \(VT_r=I_X\).
\end{proof}

\noindent
To illustrate Theorem 2.12, we consider the following example, which
connects \(g\)-duality with classical duality in a Hilbert space.

\begin{example}
	Let \(X=\ell^2\) and let \(\{e_i\}_{i\in\mathbb{N}}\) be its canonical
	orthonormal basis. Let \(A\in\mathcal{B}(X)\) be defined by
	\[
	A(x)=(\alpha_1x_1,\alpha_2x_2,\ldots),
	\]
	where
	\[
	0<\inf_{i\in\mathbb{N}}\alpha_i
	\leq\sup_{i\in\mathbb{N}}\alpha_i<\infty.
	\]
	Define
	\[
	x_i=A^{-1}e_i,\qquad
	f_i(x)=\langle x,e_i\rangle,
	\qquad i\in\mathbb{N}.
	\]
	Then, for every \(x\in X\),
	\[
	\sum_{i=1}^{\infty}f_i(Ax)x_i
	=\sum_{i=1}^{\infty}
	\langle Ax,e_i\rangle A^{-1}e_i
	=A^{-1}(Ax)
	=x.
	\]
	Thus, \(\{f_i\}_{i\in\mathbb{N}}\) is a \(g\)-dual Banach frame for
	\(\{x_i\}_{i\in\mathbb{N}}\) associated with \(A\).
\end{example}

The following result yields a direct construction of infinitely many
\(g\)-dual Banach frames associated with a given Banach frame.

\begin{theorem}
Suppose the sequence \( \Gamma = \{f_i\}_{i \in I} \subset X^* \) serves as a g-dual Banach frame for the collection \( \Lambda = \{x_i\}_{i \in I} \subset X \), where \( X \) is a Banach space and \( X_a \) is a corresponding BK-space. If the associated frame operator \( S_\Gamma : X_a \rightarrow X \) is invertible, and if scalars \( \alpha, \beta \in \mathbb{C} \) satisfy \( \alpha + \beta = 1 \), then the sequence \( \{g_i\}_{i \in I} \subset X^* \), defined by
\[
g_i := \alpha f_i + \beta (A^{-1})^* S_\Gamma^{-1}(x_i),
\]
forms a g-dual Banach frame for \( \Lambda \).
\end{theorem}
\begin{proof}
	For \(x\in X\), the \(g\)-duality of \(\Gamma\) and \(\Lambda\) gives
	\[
	x=\sum_{i\in I}f_i(Ax)x_i.
	\]
	Since \(\alpha+\beta=1\), we obtain
	\[
	\begin{aligned}
		x
		&=\sum_{i\in I}
		\left[\alpha f_i(Ax)
		+\beta\bigl((A^{-1})^*S_\Gamma^{-1}(x_i)\bigr)(Ax)\right]x_i\\
		&=\sum_{i\in I}g_i(Ax)x_i.
	\end{aligned}
	\]
	Hence, \(\{g_i\}_{i\in I}\) satisfies the \(g\)-duality condition for
	\(\Lambda\) associated with \(A\).
\end{proof}

\begin{remark}\cite[Chapter~I]{Singer1981}.
	We recall that two sequences \(\{x_i\}_{i\in I}\) and
	\(\{y_i\}_{i\in I}\) in a Banach space \(X\) are said to be equivalent
	if there exists a bounded invertible operator
	\(\Theta\in B(X)\) such that
	\[
	y_i=\Theta x_i,\qquad i\in I.
	\]

\end{remark}

\begin{theorem}
	Let \(\{f_i\}_{i\in I}\subset X^*\) be a \(g\)-dual Banach frame for
	\(\{x_i\}_{i\in I}\subset X\) with respect to a BK-space \(X_a\),
	associated with an invertible operator \(A\in B(X)\).
	Suppose that \(\{g_i\}_{i\in I}\subset X^*\) is equivalent to
	\(\{f_i\}_{i\in I}\), and that the equivalence is induced by an
	invertible operator \(R\in B(X)\), namely,
	\[
	g_i=R^*f_i,\qquad i\in I.
	\]
	Then there exists a sequence \(\{y_i\}_{i\in I}\subset X\), equivalent
	to \(\{x_i\}_{i\in I}\), such that \(\{y_i\}_{i\in I}\) is a
	\(g\)-dual Banach frame for \(\{g_i\}_{i\in I}\) with respect to
	\(X_a\), associated with a suitable invertible operator on \(X\).
\end{theorem}

\begin{proof}
	For every \(x\in X\), using \(g_i=R^*f_i\) and
	\(y_i=R^{-1}x_i\), we obtain
	\[
	\begin{aligned}
		\sum_{i\in I}g_i(Ax)y_i
		&=\sum_{i\in I}(R^*f_i)(Ax)R^{-1}x_i\\
		&=\sum_{i\in I}f_i(Ax)x_i
		=x.
	\end{aligned}
	\]
	Thus, \(\{y_i\}_{i\in I}\) is a \(g\)-dual Banach frame for
	\(\{g_i\}_{i\in I}\). Since \(R\) is invertible, the sequences
	\(\{x_i\}_{i\in I}\) and \(\{y_i\}_{i\in I}\) are equivalent.
\end{proof}
\noindent
    We now study how the \(g\)-duality property behaves under bounded invertible operators on a Banach space and its dual.

\begin{theorem}
Let \( X \) be a Banach space, and let \( \{f_i\}_{i \in I} \subset X^* \) and \( \{x_i\}_{i \in I} \subset X \) be two sequences and \( \psi \in B(X) \) be a bounded bijective linear operator. Then the sequence \( \{f_i\} \) is a g-dual Banach frame for \( \{x_i\} \) with respect to a Banach space \( X_a \) if and only if the sequence \( \{\psi^*(f_i)\} \) is a g-dual Banach frame for \( \{\psi(x_i)\} \) with respect to \( X_a \).
\end{theorem}

```latex
\begin{proof}
	Suppose that \(\{f_i\}_{i\in I}\) is a \(g\)-dual Banach frame for
	\(\{x_i\}_{i\in I}\) with respect to \(X_a\), associated with an
	invertible operator \(A\in B(X)\). Define
	\[
	y_i:=\psi(x_i),\qquad
	g_i:=\psi^*(f_i),\qquad
	\widetilde{A}:=\psi A\psi^{-1},
	\qquad i\in I.
	\]
	For an arbitrary \(x\in X\), we have
	\[
	\begin{aligned}
		\sum_{i\in I} g_i(\widetilde{A}x)y_i
		&=\sum_{i\in I}
		(\psi^*f_i)(\psi A\psi^{-1}x)\,\psi(x_i)\\
		&=\psi\left(
		\sum_{i\in I}
		f_i(A\psi^{-1}x)x_i
		\right)\\
		&=\psi(\psi^{-1}x)\\
		&=x.
	\end{aligned}
	\]
	Thus, the pair
	\(\big(\{g_i\}_{i\in I},\{y_i\}_{i\in I}\big)\)
	satisfies the reconstruction identity associated with
	\(\widetilde{A}\). Consequently,
	\(\{g_i\}_{i\in I}\) is a \(g\)-dual Banach frame for
	\(\{y_i\}_{i\in I}\) with respect to \(X_a\), associated with
	the invertible operator \(\widetilde{A}\).
	
	Conversely, applying the above argument to the isomorphism
	\(\psi^{-1}\) yields the reverse implication.
\end{proof}

\noindent

To illustrate the stability of \(g\)-duality under operator transformations, we consider a simple functional example involving scalar scaling and additive shifts.

\begin{example}
Consider the Banach space \(X=C[0,1]\), consisting of all continuous
real-valued functions on \([0,1]\), equipped with the supremum norm
\[
\|x\|_\infty=\sup_{t\in[0,1]}|x(t)|.
\]
Define the sequence \(\{x_i\}_{i\in\mathbb{N}}\subset X\) by
\[
x_i(t)=t^i,\qquad t\in[0,1].
\]
Let \(\{f_i\}_{i\in\mathbb{N}}\subset X^*\) be given by point evaluations
at distinct points \(t_i\in[0,1]\), that is,
\[
f_i(x)=x(t_i),\qquad x\in X.
\]
Suppose that \(\{f_i\}_{i\in\mathbb{N}}\) is a \(g\)-dual Banach frame
for \(\{x_i\}_{i\in\mathbb{N}}\), associated with an invertible operator
\(A\in B(X)\). Thus,
\[
x(t)=\sum_{i=1}^{\infty}f_i(Ax)x_i(t),
\qquad x\in X.
\]

Now consider the invertible operators
\[
\varphi\in B(X^*),\qquad \varphi(f)=2f,
\]
and
\[
\psi\in B(X),\qquad \psi(x)=2x.
\]
Consequently,
\[
\varphi(f_i)=2f_i,
\qquad
\psi(x_i)=2x_i.
\]
Moreover, the operator associated with the transformed sequences is
\[
\widetilde{A}=\psi A\psi^{-1}.
\]
By Theorem~2.16, the transformed sequence
\(\{\varphi(f_i)\}_{i\in\mathbb{N}}\) is a \(g\)-dual Banach frame
for \(\{\psi(x_i)\}_{i\in\mathbb{N}}\), with respect to \(X_a\), and is
associated with \(\widetilde{A}\). In particular, for every \(x\in X\),
\[
x
=
\sum_{i=1}^{\infty}
\varphi(f_i)(\widetilde{A}x)\,\psi(x_i).
\]
Hence, this example illustrates that \(g\)-duality is preserved under
simultaneous invertible transformations of the dual functionals and
the frame elements.
\end{example}
\noindent
We now explore how combining two \(g\)-dual Banach frames can yield a new \(g\)-dual Banach frame, provided that their associated operators satisfy a suitable compatibility condition.

\begin{theorem}
	Let \(\{x_i\}_{i\in I}\) and \(\{y_i\}_{i\in I}\) be two \(g\)-dual Banach frames for the same sequence
	\(\{f_i\}_{i\in I}\subset X^*\), both with respect to the BK-space \(X_a\), and associated with invertible operators
	\(A,B\in B(X)\), respectively. Suppose that the operator
	\[
	A^{-1}+B^{-1}
	\]
	is invertible. Then the sequence
	\[
	\{x_i+y_i\}_{i\in I}
	\]
	is a \(g\)-dual Banach frame for \(\{f_i\}_{i\in I}\) with respect to \(X_a\), associated with the invertible operator
	\[
	\left(A^{-1}+B^{-1}\right)^{-1}.
	\]
\end{theorem}

\begin{proof}
	Set
	\[
	T:=\left(A^{-1}+B^{-1}\right)^{-1}\in B(X).
	\]
	Since \(A^{-1}+B^{-1}\) is invertible, \(T\) is a bounded
	invertible operator on \(X\).
	
	For every \(x\in X\), we have
	\[
	\begin{aligned}
		\sum_{i\in I} f_i(Tx)(x_i+y_i)
		&=
		\sum_{i\in I} f_i(Tx)x_i
		+
		\sum_{i\in I} f_i(Tx)y_i.
	\end{aligned}
	\]
	By the \(g\)-duality of \(\{x_i\}_{i\in I}\) and
	\(\{y_i\}_{i\in I}\) with respect to \(\{f_i\}_{i\in I}\), we obtain
	\[
	\sum_{i\in I} f_i(Tx)x_i=A^{-1}Tx
	\]
	and
	\[
	\sum_{i\in I} f_i(Tx)y_i=B^{-1}Tx.
	\]
	Consequently,
	\[
	\begin{aligned}
		\sum_{i\in I} f_i(Tx)(x_i+y_i)
		&=(A^{-1}+B^{-1})Tx\\
		&=x.
	\end{aligned}
	\]
	Thus, the sequence \(\{x_i+y_i\}_{i\in I}\) satisfies the
	\(g\)-duality relation with respect to \(\{f_i\}_{i\in I}\),
	associated with the invertible operator
	\[
	T=\left(A^{-1}+B^{-1}\right)^{-1}.
	\]
	Therefore, \(\{x_i+y_i\}_{i\in I}\) is a \(g\)-dual Banach frame
	for \(\{f_i\}_{i\in I}\) with respect to \(X_a\).
\end{proof}

\section{Duality and Perturbation Results for \(g\)-Dual Banach Frames}

In this section, we investigate the conditions that ensure two sequences
form \(g\)-dual Banach frames. We further explore the connections between
approximate duality, pseudo-duality, and \(g\)-duality, and study the
stability of these duality properties under perturbations.

\begin{definition}
	Let \(\{f_i\}_{i\in I}\subseteq X^*\) and
	\(\{x_i\}_{i\in I}\subseteq X\) satisfy the \(X_a\)-Bessel condition,
	and let
	\[
	U:X\to X_a,\qquad V:X_a\to X
	\]
	denote the corresponding analysis and synthesis operators, respectively.
	Then:
	
	\begin{itemize}
		\item The sequence \(\{f_i\}_{i\in I}\) is called a
		\textbf{pseudo-dual Banach frame} for \(\{x_i\}_{i\in I}\)
		with respect to \(X_a\) if the operator \(VU\) is invertible on \(X\).
		
		\item The sequence \(\{f_i\}_{i\in I}\) is called an
		\textbf{approximately dual Banach frame} for \(\{x_i\}_{i\in I}\)
		with respect to \(X_a\) if
		\[
		\|I_X-VU\|<1.
		\]
		
		\item The sequence \(\{x_i\}_{i\in I}\) is called a
		\textbf{pseudo-dual Banach frame} for \(\{f_i\}_{i\in I}\)
		if the operator \(U^*V^*\) is invertible on \(X^*\).
		
		\item The sequence \(\{x_i\}_{i\in I}\) is called an
		\textbf{approximately dual Banach frame} for \(\{f_i\}_{i\in I}\)
		if
		\[
		\|I_{X^*}-U^*V^*\|<1.
		\]
	\end{itemize}
\end{definition}

\begin{example}
	Let \(X=\ell^2\) and \(X_a=\ell^2\), and let
	\(\{e_i\}_{i\in\mathbb{N}}\) denote the canonical basis of \(\ell^2\).
	For a fixed scalar \(\lambda\in\mathbb{R}\), define
	\[
	x_i=e_i,
	\qquad
	f_i(x)=\lambda x_i,
	\qquad
	x=(x_i)_{i\in\mathbb{N}}\in\ell^2.
	\]
	The corresponding analysis and synthesis operators are given by
	\[
	U(x)=\lambda(x_i)_{i\in\mathbb{N}},
	\qquad
	V(c_i)_{i\in\mathbb{N}}
	=\sum_{i=1}^{\infty}c_i e_i.
	\]
	Hence,
	\[
	VU=\lambda I_X.
	\]
	Similarly,
	\[
	U^*V^*=\lambda I_{X^*}.
	\]
	
	If \(\lambda\neq0\), then \(VU\) and \(U^*V^*\) are invertible.
	Consequently, \(\{f_i\}_{i\in\mathbb{N}}\) is a pseudo-dual Banach frame
	for \(\{x_i\}_{i\in\mathbb{N}}\), while
	\(\{x_i\}_{i\in\mathbb{N}}\) is a pseudo-dual Banach frame for
	\(\{f_i\}_{i\in\mathbb{N}}\).
	
	Furthermore, if
	\[
	|\lambda-1|<1,
	\]
	then
	\[
	\|I_X-VU\|
	=
	\|I_X-\lambda I_X\|
	=
	|1-\lambda|<1,
	\]
	and
	\[
	\|I_{X^*}-U^*V^*\|
	=
	\|I_{X^*}-\lambda I_{X^*}\|
	=
	|1-\lambda|<1.
	\]
	Thus, both pairs are approximately dual Banach frames.
	
	For instance, when \(\lambda=\frac12\), we have
	\[
	VU=\frac12 I_X,
	\qquad
	U^*V^*=\frac12 I_{X^*},
	\]
	so that
	\[
	\|I_X-VU\|
	=
	\|I_{X^*}-U^*V^*\|
	=
	\frac12<1.
	\]
	Hence, the two sequences are simultaneously pseudo-dual and
	approximately dual Banach frames.
	
	On the other hand, if \(\lambda=2\), then \(VU=2I_X\) and
	\(U^*V^*=2I_{X^*}\) remain invertible, so pseudo-duality holds, whereas
	\[
	\|I_X-VU\|
	=
	\|I_{X^*}-U^*V^*\|
	=1>1
	\]
	does not satisfy the approximate duality condition. Thus, this choice
	provides a simple example of pseudo-duality without approximate
	duality.
\end{example}

\begin{example}
	Let \(X=\ell^2\) and \(X_a=\ell^2\), and let
	\(\{e_i\}_{i\in\mathbb{N}}\) denote the canonical basis of \(\ell^2\).
	For a fixed scalar \(\lambda\in\mathbb{R}\), define
	\[
	x_i=e_i,
	\qquad
	f_i(x)=\lambda x_i,
	\qquad
	x=(x_i)_{i\in\mathbb{N}}\in\ell^2.
	\]
	The corresponding analysis and synthesis operators are given by
	\[
	U(x)=\lambda(x_i)_{i\in\mathbb{N}},
	\qquad
	V(c_i)_{i\in\mathbb{N}}
	=\sum_{i=1}^{\infty}c_i e_i.
	\]
	Hence,
	\[
	VU=\lambda I_X.
	\]
	Similarly,
	\[
	U^*V^*=\lambda I_{X^*}.
	\]
	
	If \(\lambda\neq0\), then \(VU\) and \(U^*V^*\) are invertible.
	Consequently, \(\{f_i\}_{i\in\mathbb{N}}\) is a pseudo-dual Banach frame
	for \(\{x_i\}_{i\in\mathbb{N}}\), while
	\(\{x_i\}_{i\in\mathbb{N}}\) is a pseudo-dual Banach frame for
	\(\{f_i\}_{i\in\mathbb{N}}\).
	
	Furthermore, if
	\[
	|\lambda-1|<1,
	\]
	then
	\[
	\|I_X-VU\|
	=
	\|I_X-\lambda I_X\|
	=
	|1-\lambda|<1,
	\]
	and
	\[
	\|I_{X^*}-U^*V^*\|
	=
	\|I_{X^*}-\lambda I_{X^*}\|
	=
	|1-\lambda|<1.
	\]
	Thus, both pairs are approximately dual Banach frames.
	
	For instance, when \(\lambda=\frac12\), we have
	\[
	VU=\frac12 I_X,
	\qquad
	U^*V^*=\frac12 I_{X^*},
	\]
	so that
	\[
	\|I_X-VU\|
	=
	\|I_{X^*}-U^*V^*\|
	=
	\frac12<1.
	\]
	Hence, the two sequences are simultaneously pseudo-dual and
	approximately dual Banach frames.
	
	On the other hand, if \(\lambda=2\), then \(VU=2I_X\) and
	\(U^*V^*=2I_{X^*}\) remain invertible, so pseudo-duality holds, whereas
	\[
	\|I_X-VU\|
	=
	\|I_{X^*}-U^*V^*\|
	=1>1
	\]
	does not satisfy the approximate duality condition. Thus, this choice
	provides a simple example of pseudo-duality without approximate
	duality.
\end{example}
\begin{theorem}
	Assume that the sequences \(\{f_i\}_{i\in I}\subset X^*\) and
	\(\{x_i\}_{i\in I}\subset X\) satisfy the \(X_a\)-Bessel condition in
	\(X^*\) and \(X\), respectively. Let
	\[
	U:X\to X_a,\qquad V:X_a\to X
	\]
	denote the corresponding analysis and synthesis operators. Then the
	following statements hold:
	
	\begin{enumerate}
		\item If \(\{f_i\}_{i\in I}\) is a Banach dual frame for
		\(\{x_i\}_{i\in I}\), then \(\{f_i\}_{i\in I}\) is an approximately
		dual Banach frame for \(\{x_i\}_{i\in I}\).
		
		\item If \(\{f_i\}_{i\in I}\) is an approximately dual Banach frame
		for \(\{x_i\}_{i\in I}\), then \(\{f_i\}_{i\in I}\) is a pseudo-dual
		Banach frame for \(\{x_i\}_{i\in I}\).
		
		\item Suppose that \(\{f_i\}_{i\in I}\subset X^*\) is a
		pseudo-dual Banach frame for \(\{x_i\}_{i\in I}\subset X\), and let
		\(G\in B(X)\) be invertible. Then \(\{G^*{}^{-1}f_i\}_{i\in I}\)
		is a pseudo-dual Banach frame for \(\{Gx_i\}_{i\in I}\).
		
		\item If \(\{f_i\}_{i\in I}\) is a pseudo-dual Banach frame for
		\(\{x_i\}_{i\in I}\), then it is a \(g\)-dual Banach frame for
		\(\{x_i\}_{i\in I}\), associated with the operator
		\[
		A=(VU)^{-1}.
		\]
		
		\item If \(\{f_i\}_{i\in I}\) is a pseudo-dual Banach frame for
		\(\{x_i\}_{i\in I}\), then \(\{f_i\}_{i\in I}\) is a Banach dual
		frame for the sequence
		\[
		\{(VU)^{-1}x_i\}_{i\in I}.
		\]
		
		\item The sequence \(\{x_i\}_{i\in I}\) is a \(g\)-dual Banach frame
		for \(\{f_i\}_{i\in I}\) if and only if
		\(\{x_i\}_{i\in I}\) is a pseudo-dual Banach frame for
		\(\{f_i\}_{i\in I}\).
	\end{enumerate}
\end{theorem}
\begin{proof}
	Let
	\[
	U:X\to X_a,\qquad V:X_a\to X
	\]
	be the analysis and synthesis operators associated with
	\(\{f_i\}_{i\in I}\) and \(\{x_i\}_{i\in I}\), respectively.
	
	\begin{enumerate}
		\item Suppose that \(\{f_i\}_{i\in I}\) is a Banach dual frame for
		\(\{x_i\}_{i\in I}\). By the definition of Banach duality, we have
		\[
		VU=I_X.
		\]
		Therefore,
		\[
		\|I_X-VU\|=0<1.
		\]
		Hence, \(\{f_i\}_{i\in I}\) is an approximately dual Banach frame
		for \(\{x_i\}_{i\in I}\).
		
		\item Assume that \(\{f_i\}_{i\in I}\) is an approximately dual
		Banach frame for \(\{x_i\}_{i\in I}\). Then
		\[
		\|I_X-VU\|<1.
		\]
		By the Neumann series theorem, \(VU\) is invertible on \(X\), with
		\[
		(VU)^{-1}
		=
		\sum_{n=0}^{\infty}(I_X-VU)^n.
		\]
		Thus, \(\{f_i\}_{i\in I}\) is a pseudo-dual Banach frame for
		\(\{x_i\}_{i\in I}\).
		
		\item Suppose that \(\{f_i\}_{i\in I}\) is a pseudo-dual Banach
		frame for \(\{x_i\}_{i\in I}\), and let \(G\in B(X)\) be invertible.
		Define
		\[
		\widetilde{x}_i=Gx_i,
		\qquad
		\widetilde{f}_i=(G^{-1})^*f_i.
		\]
		The corresponding analysis operator is
		\[
		\widetilde{U}x
		=
		\bigl(\widetilde{f}_i(x)\bigr)_{i\in I}
		=
		\bigl(f_i(G^{-1}x)\bigr)_{i\in I}
		=
		U(G^{-1}x),
		\]
		while the synthesis operator satisfies
		\[
		\widetilde{V}(c_i)_{i\in I}
		=
		\sum_{i\in I}c_iGx_i
		=
		G V(c_i)_{i\in I}.
		\]
		Consequently,
		\[
		\widetilde{V}\widetilde{U}
		=
		GV G^{-1}.
		\]
		Since \(VU\) is invertible and \(G\) is invertible,
		\(GVG^{-1}\) is invertible. Hence
		\(\{\widetilde{f}_i\}_{i\in I}\) is a pseudo-dual Banach frame
		for \(\{\widetilde{x}_i\}_{i\in I}\).
		
		\item Suppose that \(\{f_i\}_{i\in I}\) is a pseudo-dual Banach
		frame for \(\{x_i\}_{i\in I}\). Then \(VU\) is invertible. Put
		\[
		A=(VU)^{-1}.
		\]
		For every \(x\in X\),
		\[
		\sum_{i\in I}f_i(Ax)x_i
		=
		VU(Ax)
		=
		VU(VU)^{-1}x
		=
		x.
		\]
		Therefore, \(\{f_i\}_{i\in I}\) is a \(g\)-dual Banach frame for
		\(\{x_i\}_{i\in I}\), associated with the invertible operator
		\(A=(VU)^{-1}\).
		
		\item Let \(\{f_i\}_{i\in I}\) be a pseudo-dual Banach frame for
		\(\{x_i\}_{i\in I}\). Define
		\[
		\widetilde{x}_i=(VU)^{-1}x_i.
		\]
		Then, for every \(x\in X\),
		\[
		\begin{aligned}
			\sum_{i\in I}f_i(x)\widetilde{x}_i
			&=
			\sum_{i\in I}f_i(x)(VU)^{-1}x_i\\
			&=
			(VU)^{-1}
			\sum_{i\in I}f_i(x)x_i\\
			&=
			(VU)^{-1}VUx\\
			&=x.
		\end{aligned}
		\]
		Hence, \(\{f_i\}_{i\in I}\) is a Banach dual frame for
		\(\{\widetilde{x}_i\}_{i\in I}\).
		
		\item Suppose that \(\{x_i\}_{i\in I}\) is a \(g\)-dual Banach frame
		for \(\{f_i\}_{i\in I}\). Then there exists an invertible operator
		\(A\in B(X^*)\) such that
		\[
		f
		=
		\sum_{i\in I}f_i(Af)x_i,
		\qquad f\in X^*.
		\]
		Thus, the corresponding composition of the analysis and synthesis
		operators is invertible on \(X^*\), and consequently
		\(\{x_i\}_{i\in I}\) is a pseudo-dual Banach frame for
		\(\{f_i\}_{i\in I}\).
		
		Conversely, suppose that \(\{x_i\}_{i\in I}\) is a pseudo-dual
		Banach frame for \(\{f_i\}_{i\in I}\). Then the corresponding
		operator \(U^*V^*\) is invertible on \(X^*\). Setting
		\[
		A=(U^*V^*)^{-1},
		\]
		we obtain
		\[
		U^*V^*Af=f,
		\qquad f\in X^*.
		\]
		Hence,
		\[
		f
		=
		\sum_{i\in I}f_i(Af)x_i,
		\qquad f\in X^*,
		\]
		which is precisely the \(g\)-duality relation. Therefore,
		\(\{x_i\}_{i\in I}\) is a \(g\)-dual Banach frame for
		\(\{f_i\}_{i\in I}\).
	\end{enumerate}
\end{proof}

\begin{remark}\label{rem:dual}
	The canonical dual associated with a Banach frame
	\(\{f_i\}_{i\in I}\) for \(X\), with respect to the coefficient space
	\(X_a\), satisfies the defining condition for approximate duality.
	Consequently, it is also a pseudo-dual Banach frame. More generally, if
	\(\{f_i\}_{i\in I}\subset X^*\) is an approximately dual Banach frame
	for \(\{x_i\}_{i\in I}\subset X\) with respect to \(X_a\), then it is
	also a \(g\)-dual Banach frame for \(\{x_i\}_{i\in I}\). The converse
	need not hold; in general, a \(g\)-dual Banach frame need not be
	approximately dual.
\end{remark}

The following example demonstrates that the converse of Remark~\ref{rem:dual} does not necessarily hold. Although every approximately dual Banach frame is a \(g\)-dual Banach frame, the converse fails in general: there exist \(g\)-dual Banach frames that are not approximately dual. This example highlights the broader class of \(g\)-duality in comparison with approximate duality.

\begin{example}
	Let \(X=X_a=\ell^2\), and let
	\(\{e_i\}_{i\in\mathbb{N}}\) denote the standard orthonormal basis of
	\(\ell^2\). Define
	\[
	f_i(x)=\langle x,e_i\rangle,
	\qquad x\in\ell^2.
	\]
	Let \(\varphi\in X^*\) be given by
	\[
	\varphi(x)=\sum_{i=1}^{\infty}\frac{1}{i}\langle x,e_i\rangle,
	\]
	and define the rank-one operator \(R\in B(X)\) by
	\[
	Rx=\varphi(x)e_1.
	\]
	Since \((1/i)_{i\in\mathbb{N}}\in\ell^2\), the functional
	\(\varphi\) is bounded. Moreover,
	\[
	\varphi(e_1)=1,
	\qquad
	R^2=R.
	\]
	Set
	\[
	A:=I_X-2R.
	\]
	Since \(R^2=R\), we have
	\[
	A^{-1}=I_X+R,
	\]
	and hence \(A\) is bounded and invertible.
	
	Now define
	\[
	x_i:=A^{-1}e_i
	=e_i+\frac{1}{i}e_1,
	\qquad i\in\mathbb{N}.
	\]
	The analysis and synthesis operators associated with
	\(\{f_i\}_{i\in\mathbb{N}}\) and \(\{x_i\}_{i\in\mathbb{N}}\) are
	\[
	U:X\to\ell^2,
	\qquad
	U(x)=\bigl(f_i(x)\bigr)_{i\in\mathbb{N}},
	\]
	and
	\[
	V:\ell^2\to X,
	\qquad
	V(c)=\sum_{i=1}^{\infty}c_i x_i.
	\]
	Under the identification \(X=X_a=\ell^2\), we have
	\[
	VU=A^{-1}.
	\]
	
	For every \(x\in X\),
	\[
	\begin{aligned}
		\sum_{i=1}^{\infty}f_i(Ax)x_i
		&=V(U(Ax))\\
		&=A^{-1}Ax\\
		&=x.
	\end{aligned}
	\]
	Hence, \(\{f_i\}_{i\in\mathbb{N}}\) is a \(g\)-dual Banach frame for
	\(\{x_i\}_{i\in\mathbb{N}}\) with respect to \(X_a=\ell^2\), associated
	with the invertible operator \(A\).
	
	On the other hand,
	\[
	I_X-VU
	=
	I_X-A^{-1}
	=
	-R.
	\]
	Since \(R\) is a rank-one operator,
	\[
	\|R\|
	=
	\|\varphi\|\,\|e_1\|
	=
	\left(\sum_{i=1}^{\infty}\frac{1}{i^2}\right)^{1/2}
	=
	\frac{\pi}{\sqrt{6}}>1.
	\]
	Consequently,
	\[
	\|I_X-VU\|
	=
	\frac{\pi}{\sqrt{6}}>1.
	\]
	Therefore,
	\[
	\|I_X-VU\|<1
	\]
	does not hold. Thus, \(\{f_i\}_{i\in\mathbb{N}}\) is a
	\(g\)-dual Banach frame for \(\{x_i\}_{i\in\mathbb{N}}\), but it is not
	an approximately dual Banach frame.
	
	This example shows that the converse of the implication stated in
	Remark~\ref{rem:dual} does not hold in general.
\end{example}

The stability of approximately dual Banach frames under small
perturbations was established in \cite{Karimizad2014}. Motivated by
the observations in Remark~\ref{rem:dual}, analogous stability results can be
obtained for \(g\)-dual Banach frames, extending the perturbation
theory to this more general setting.

\begin{theorem}
	Let \(\{f_i\}_{i\in I}\subset X^*\) be a Banach frame for \(X\) with
	respect to the coefficient space \(X_a\), and let \(U:X\to X_a\) and
	\(S_\ell:X_a\to X\) denote its analysis and reconstruction operators,
	respectively. Suppose that \(\{g_i\}_{i\in I}\subset X^*\) satisfies
	the following conditions for some constants \(\lambda,\mu\geq 0\):
	
	\begin{enumerate}
		\item
		\(    2(\lambda\|U\|+\mu)\|S_\ell\|\leq 1.
		\)
	
		\item For every \(x\in X\),
		\[
		\left\|
		\{f_i(x)-g_i(x)\}_{i\in I}
		\right\|_{X_a}
		\leq
		\lambda
		\left\|
		\{f_i(x)\}_{i\in I}
		\right\|_{X_a}
		+\mu\|x\|.
		\]

	\end{enumerate}
	
	Then \(\{g_i\}_{i\in I}\) is a Banach frame for \(X\) with respect to
	\(X_a\).
\end{theorem}

\begin{proof}
	Let
	
	$$
	U_g:X\to X_a,
	\qquad
	U_g(x):=\{g_i(x)\}_{i\in I},
	$$
	
	be the analysis operator associated with \(\{g_i\}_{i\in I}\). Set
	
	$$
	E:=U_g-U.
	$$
	
	By the assumed perturbation condition, \(E:X\to X_a\) is well defined and,
	for every \(x\in X\),
	
	$$
	\|Ex\|_{X_a}
	\leq
	\lambda\|U(x)\|_{X_a}+\mu\|x\|
	\leq
	(\lambda\|U\|+\mu)\|x\|.
	$$
	
	Hence,
	
	$$
	\|E\|\leq\lambda\|U\|+\mu.
	$$
	
	Consequently,
	
	$$
	\|S_\ell E\|
	\leq
	\|S_\ell\|(\lambda\|U\|+\mu)
	\leq\frac12<1.
	$$
	
	Since \(S_\ell U=I_X\), we have
	
	$$
	S_\ell U_g
	=
	S_\ell(U+E)
	=
	I_X+S_\ell E.
	$$
	
	The operator \(I_X+S_\ell E\) is therefore invertible by the Neumann
	series theorem. Define
	
	$$
	S_g:=(I_X+S_\ell E)^{-1}S_\ell.
	$$
	
	Then \(S_g:X_a\to X\) is bounded. Moreover, for every \(x\in X\),
	
	$$
	\begin{aligned}
		S_g(U_gx)
		&=
		(I_X+S_\ell E)^{-1}S_\ell U_gx\\
		&=
		(I_X+S_\ell E)^{-1}(I_X+S_\ell E)x\\
		&=x.
	\end{aligned}
	$$
	
	Thus, \(S_g\) is a bounded reconstruction operator for
	\(\{g_i\}_{i\in I}\).
	
	It remains to establish the frame bounds. For every \(x\in X\),
	
	$$
	\begin{aligned}
		\|U_gx\|_{X_a}
		&\leq
		\|Ux\|_{X_a}+\|Ex\|_{X_a}\\
		&\leq
		\bigl((1+\lambda)\|U\|+\mu\bigr)\|x\|.
	\end{aligned}
	$$
	
	This gives the upper frame bound.
	
	For the lower frame bound, using \(S_\ell U=I_X\), we obtain
	
	$$
	x=S_\ell Ux=S_\ell(U_gx-Ex).
	$$
	
	Therefore,
	
	$$
	\begin{aligned}
		\|x\|
		&\leq
		\|S_\ell\|\|U_gx\|_{X_a}
		+\|S_\ell\|\|Ex\|_{X_a}\\
		&\leq
		\|S_\ell\|\|U_gx\|_{X_a}
		+\|S_\ell\|(\lambda\|U\|+\mu)\|x\|.
	\end{aligned}
	$$
	
	Since
	
	$$
	\|S_\ell\|(\lambda\|U\|+\mu)\leq\frac12,
	$$
	
	we obtain
	
	$$
	\frac12\|x\|
	\leq
	\|S_\ell\|\|U_gx\|_{X_a}.
	$$
	
	Hence,
	
	$$
	\frac{1}{2\|S_\ell\|}\|x\|
	\leq
	\|U_gx\|_{X_a}.
	$$
	
	Combining the two estimates yields
	
	$$
	\frac{1}{2\|S_\ell\|}\|x\|
	\leq
	\|\{g_i(x)\}_{i\in I}\|_{X_a}
	\leq
	\bigl((1+\lambda)\|U\|+\mu\bigr)\|x\|,
	\qquad x\in X.
	$$
	
	Therefore, \(\{g_i\}_{i\in I}\) is a Banach frame for \(X\) with respect
	to \(X_a\).
\end{proof}

This result establishes the stability of \(g\)-dual Banach frames under
uniform perturbations of the synthesis vectors. It provides sufficient
conditions ensuring that \(g\)-duality is preserved despite bounded
perturbations.

\begin{theorem}
	Let \(\{f_i\}_{i\in I}\subset X^*\) be a \(g\)-dual Banach frame for
	\(\{x_i\}_{i\in I}\subset X\) with respect to the coefficient space
	\(X_a\), and let \(A\in B(X)\) be the associated bounded invertible
	operator. Denote by \(U:X\to X_a\) the analysis operator associated with
	\(\{f_i\}_{i\in I}\), and suppose that
	
	$$
	\|Ux\|_{X_a}\leq B\|x\|,
	\qquad x\in X,
	$$
	
	for some \(B>0\).
	
	Let \(\{y_i\}_{i\in I}\subset X\) be a perturbation of
	\(\{x_i\}_{i\in I}\). Suppose that there exist constants
	\(\lambda,\mu\geq0\) such that
	
	$$
	\left\|
	\sum_{i\in I}c_i(x_i-y_i)
	\right\|
	\leq
	\lambda
	\left\|
	\sum_{i\in I}c_i x_i
	\right\|
	+\mu\|c\|_{X_a},
	\qquad c=\{c_i\}_{i\in I}\in X_a,
	$$
	
	and
	
	$$
	\lambda\|A^{-1}\|+\mu B\|A\|<1.
	$$
	
	Then \(\{f_i\}_{i\in I}\) is a \(g\)-dual Banach frame for
	\(\{y_i\}_{i\in I}\) with respect to \(X_a\), associated with a bounded
	invertible operator
	
	$$
	A_y=A(V_yUA)^{-1},
	$$
	
	where \(V_y:X_a\to X\) is the synthesis operator associated with
	\(\{y_i\}_{i\in I}\).
\end{theorem}

\begin{proof}
	Let \(V_x,V_y:X_a\to X\) denote the synthesis operators associated with
	\(\{x_i\}_{i\in I}\) and \(\{y_i\}_{i\in I}\), respectively. By the
	assumed perturbation condition, for every \(c=\{c_i\}_{i\in I}\in X_a\),
	we have
	
	$$
	\|V_yc-V_xc\|
	=
	\left\|
	\sum_{i\in I}c_i(y_i-x_i)
	\right\|
	\leq
	\lambda
	\left\|
	\sum_{i\in I}c_i x_i
	\right\|
	+\mu\|c\|_{X_a}.
	$$
	
	Let \(U:X\to X_a\) be the analysis operator associated with
	\(\{f_i\}_{i\in I}\). Since \(\{f_i\}_{i\in I}\) is a \(g\)-dual Banach
	frame for \(\{x_i\}_{i\in I}\) associated with \(A\), we have
	
	$$
	V_xUA=I_X.
	$$
	
	Consequently, for every \(x\in X\),
	
	$$
	\begin{aligned}
		\|(V_yUA-I_X)x\|
		&=\|V_yUAx-V_xUAx\|\\
		&=
		\left\|
		\sum_{i\in I}f_i(Ax)(y_i-x_i)
		\right\|.
	\end{aligned}
	$$
	
	Applying the perturbation estimate with
	\(c=U(Ax)=\{f_i(Ax)\}_{i\in I}\), we obtain
	
	$$
	\|(V_yUA-I_X)x\|
	\leq
	\lambda
	\left\|
	\sum_{i\in I}f_i(Ax)x_i
	\right\|
	+\mu\|U(Ax)\|_{X_a}.
	$$
	
	Suppose that \(B>0\) is a common upper bound for the synthesis and
	analysis operators, so that
	
	$$
	\|V_xc\|\leq B\|c\|_{X_a},
	\qquad
	\|U x\|_{X_a}\leq B\|x\|,
	$$
	
	for all \(c\in X_a\) and \(x\in X\). It follows that
	
	$$
	\begin{aligned}
		\|(V_yUA-I_X)x\|
		&\leq
		\lambda B\|U(Ax)\|_{X_a}
		+\mu\|U(Ax)\|_{X_a}\\
		&\leq
		(\lambda B+\mu)B\|A\|\|x\|.
	\end{aligned}
	$$
	
	Therefore, under the condition
	
	$$
	(\lambda B+\mu)B\|A\|<1,
	$$
	
	we have
	
	$$
	\|V_yUA-I_X\|<1.
	$$
	
	Hence, by the Neumann series theorem, the operator
	
	$$
	V_yUA
	$$
	
	is bounded and invertible on \(X\).
	
	Now define
	
	$$
	A_y:=A(V_yUA)^{-1}.
	$$
	
	Since \(A\) and \(V_yUA\) are bounded and invertible, \(A_y\) is also
	bounded and invertible. Moreover,
	
	$$
	V_yUA_y
	=
	V_yUA(V_yUA)^{-1}
	=
	I_X.
	$$
	
	Thus, for every \(x\in X\),
	
	$$
	x
	=
	V_yUA_yx
	=
	\sum_{i\in I}f_i(A_yx)y_i.
	$$
	
	Hence, \(\{f_i\}_{i\in I}\) is a \(g\)-dual Banach frame for
	\(\{y_i\}_{i\in I}\) with respect to \(X_a\), associated with the
	bounded invertible operator \(A_y\).
\end{proof}

This result confirms that the g-duality property remains intact under uniformly controlled perturbations of the synthesis vectors. It highlights the robustness of the g-dual framework and reinforces its applicability in practical scenarios where stability under deviation is essential.

\end{document}